\documentclass[12pt]{amsart}
\makeatletter

\theoremstyle{plain}

\newtheorem{thm}{Theorem}[section]

\newtheorem{cor}[thm]{Corollary}

\newtheorem{lem}[thm]{Lemma}
\newtheorem{prop}[thm]{Proposition}
\theoremstyle{definition}
\newtheorem{defn}[thm]{Definition}
\theoremstyle{remark}
\newtheorem{rem}[thm]{Remark}
\theoremstyle{remark}

\theoremstyle{remark}

\theoremstyle{remark}

\theoremstyle{definition}
\newtheorem{example}[thm]{Example}

\theoremstyle{definition}

\theoremstyle{plain}

\theoremstyle{definition}

\theoremstyle{remark}

\theoremstyle{remark}

\theoremstyle{definition}

\theoremstyle{remark}
\newtheorem*{acknowledgement*}{Acknowledgement}

\newcommand{\C}{{\mathbb C}}
\newcommand{\N}{{\mathbb N}}
\newcommand{\Z}{{\mathbb Z}}

\newcommand{\K}{{\mathbb K}}

\newcommand{\e}{\varepsilon}

\newcommand{\hh}{{\mathcal H}}

\newcommand{\hm}{{\mathcal M}}

\newcommand{\hr}{{\mathcal R}}
\newcommand{\id}{\mathrm{id}}

\DeclareMathOperator{\Ad}{Ad}

\DeclareMathOperator{\Aut}{Aut}

\newcommand{\T}{\mathbb{T}}

\newcommand{\rdim}{\mathrm{dim}_{\mathrm{Rok}}}

\def\freeprod{\font\bigsymbolsfont=cmsy10 scaled \magstep3
 \setbox0=\hbox{\bigsymbolsfont\char'003 }\mathord{\lower1pt\box0}}\relax\ignorespaces

\newcommand{\Hawaii}{Hawai\kern.05em`\kern.05em\relax i}

\usepackage{amssymb} 
\usepackage{amscd} 
\usepackage{hyperref}  
\usepackage{tikz} 
\usetikzlibrary{matrix,arrows}

\makeatother

\begin{document}

\title{Rokhlin dimension of $\Z^m$-actions on simple $C^*$-algebras}

\author{Hung-Chang Liao}

\address{Fachbereich Mathematik, Westf\"alische Wilhelms-Universität M\"unster, 48149 M\"unster, Germany}


\email{liao@uni-muenster.de}

\begin{abstract} We study Rokhlin dimension of $\Z^m$-actions on simple separable stably finite nuclear $C^*$-algebras. We prove that under suitable assumptions, a strongly outer $\Z^m$-action has finite Rokhlin dimension. This extends the known result for automorphisms. As an application, we show that for a large class of $C^*$-algebras, the $\Z^m$-Bernoulli action has finite Rokhlin dimension. 
\end{abstract}

\maketitle

\section{Introduction}
Extending topological covering dimension to the context of nuclear $C^*$-algebras, Winter and Zacharias introduced \emph{nuclear dimension} in \cite{WZ09}. Recent advances in the study of nuclear $C^*$-algebras have revealed that: 
\begin{itemize}
	\item[(1)] Finiteness of nuclear dimension implies other structural properties, including Jiang-Su stability and strict comparison (\cite{Win12, Ror04}). In fact, a conjecture of Toms and Winter predicts that these are all equivalent for unital separable simple nuclear $C^*$-algebras. 
	\item[(2)] Unital separable simple $C^*$-algebras of finite nuclear dimension have been completely classified by $K$-theoretic invariants, provided they satisfy a technical condition known as the UCT (Universal Coefficient Theorem) (\cite{Kir, Phi00, TWW17, EGLN15}).
\end{itemize} 
Given the success above, it is important to understand how nuclear dimension behaves under various constructions, such as crossed products. Motivated by the Kakutani-Rokhlin lemma in ergodic theory (\cite{Kak43, Roh48}), its non-commutative analogues for finite von Neumann algebras (\cite{Con75, Ocn85}) and later developments in the $C^*$-realm, Hirshberg, Winter, and Zacharias introduced \emph{Rokhlin dimension} in \cite{HWZ15} for actions of finite groups and the set of integers $\Z$ on unital $C^*$-algebras. One of the main results in \cite{HWZ15} is that finiteness of nuclear dimension passes to crossed products when the action has finite Rokhlin dimension. In addition, they showed that:
\begin{itemize}
	\item[(1)] Actions with finite Rokhlin dimension are ubiquitous; more precisely, given a unital Jiang-Su stable $C^*$-algebra $A$, the automorphisms with finite Rokhlin dimension form a $G_\delta$-subset of the automorphism group $\Aut(A)$ (\cite[Theorem 3.4]{HWZ15}).
	\item[(2)] There are plenty of natural examples coming from actions on commutative $C^*$-algebras, i.e., topological dynamics. In fact, given any minimal homeomorphism of a compact metrizable space with finite covering dimension, the induced $C^*$-system has finite Rokhlin dimension (\cite[Theorem 6.1]{HWZ15}).
\end{itemize}
Since the initial paper of Hirshberg, Winter, and Zacharias, the notion of Rokhlin dimension has been generalized to many different contexts. Rokhlin dimension for actions of $\Z^m$ was introduced by Szab\'o in \cite{Sza15}, and along the same line  Szab\'o, Wu, and Zacharias in \cite{SWZ15} defined Rokhlin dimension for actions of residually finite groups. Similar results regarding nuclear dimension, genericity, and topological dynamics were also obtained. The other developments include Rokhlin dimension for compact groups actions \cite{Gar14}), for $C^*$-correspondences \cite{BTZ16}, for flows \cite{HSWW16}, and for compact quantum group actions \cite{BSV17, GKL17}.

The present paper focuses on $C^*$-dynamical systems of a purely non-commutative nature. We consider $\Z^m$-actions on simple separable stably finite nuclear (hence tracial) $C^*$-algebras. Our first main result is a Rokhlin type theorem: under suitable assumptions, a certain strong form of outerness, which is a necessary condition for the action to have finite Rokhlin dimension (see Proposition \ref{prop:main-converse}), turns out to be sufficient. The precise statement is given below (see Section 2 and Section 5 for relevant definitions) Recall that \emph{Property (SI)}, as introduced by Sato in \cite{Sat10}, is a comparison-by-traces type property for central sequence algebras (see Definition \ref{defn:properptySI} for the precise definition). This property is possessed by a large class of simple nuclear $C^*$-algebras, in particular the ones of finite nuclear dimension (cf. \cite{MS12a, Win12, Ror04}).

\begin{thm} \label{thm:1-1} (Theorem \ref{thm:main})
Let $A$ be a unital simple separable nuclear $C^*$-algebra with nonempty trace space $T(A)$, and let $\alpha:\Z^m\to \Aut(A)$ be a group action. Suppose
\begin{enumerate}
	\item[(1)] $A$ has property (SI),
	\item[(2)] $\alpha$ is strongly outer,
	\item[(3)] $T(A)$ is a Bauer simplex with finite dimensional extreme boundary,
	\item[(4)] $\tau\circ \alpha = \tau$ for every $\tau\in T(A)$.
\end{enumerate}
Then $\rdim(\alpha) \leq 4^m-1$.
\end{thm}

Note that on the purely infinite side, G\'abor Szab\'o has shown that for actions of countable discrete residually finite amenable groups on Kirchberg algebras, pointwise outerness implies finite Rokhlin dimension \cite{Sza16}. On the stably finite side, Eusebio Gardella and Ilan Hirshberg have been able to prove a result similar to Theorem \ref{thm:1-1} for finite group actions, without any restrction on the induced action on the tracial simplex.

This type of results dates back to Connes' non-commutative Rokhlin lemma \cite{Con76}, where he showed that outer $\Z$-actions on finite von Neumann algebras have the ($W^*$)-Rokhlin property. Ocneanu in \cite{Ocn85} obtained a similar result for actions of discrete amenable groups. On the $C^*$-side, Hermann and Ocneanu introduced a Rokhlin property for automorphisms of $C^*$-algebras \cite{HO84}. The definition was later refined by Kishimoto \cite{Kis95}, who proved that for certain AF and A$\T$ algebras, strong outerness implies the Rokhlin property \cite{Kis96, Kis98}. Nakamura studied $\Z^2$-actions on UHF algebras and obtained a similar Rokhlin type result \cite{Nak99}. Various Rokhlin type definitions and theorems also appeared in the work of Osaka and Phillips \cite{OP06a, OP06b}, Sato \cite{Sat10}, and Matui and Sato \cite{MS12b, MS14}. Recently, we proved a Rokhlin type theorem similar to Theorem \ref{thm:1-1} for automorphisms \cite{Lia16}, which was built upon a remarkable breakthrough made by Matui and Sato (unpublished; see \cite[Section 6]{Lia16}).

In Section 2 we will prove the converse of Theorem \ref{thm:1-1} (see Proposition \ref{prop:main-converse}). Combining this and Remark \ref{rem:trace}, we obtain the following corollary:
\begin{cor}
Let $A$ be a unital simple separable nuclear $C^*$-algebra with nonempty trace space $T(A)$, and let $\alpha: \Z^m\to \Aut(A)$ be a group action. Suppose
\begin{enumerate}
	\item[(1)] $A$ has property (SI),
	\item[(2)] $T(A)$ is a Bauer simplex with finite dimensional extreme boundary, and
	\item[(3)] $\tau\circ \alpha = \tau$ for every $\tau\in T(A)$.
\end{enumerate}
Then the following are equivalent.
\begin{enumerate}
	\item[(i)] $\dim_{Rok}(\alpha) \leq 4^m-1$.
	\item[(ii)] For every $v\in \Z^m\setminus \{0\}$, the restriction map $r:T( A \rtimes_{\alpha^v}\Z  )\to T(A)^{\alpha^v}$ is injective.
	\item[(iii)] The action $\alpha$ is strongly outer.
\end{enumerate}
\end{cor}
\noindent While in general it is not easy to verify that a given action has finite Rokhlin dimension, strong outerness can often be checked explicitly. As an example, we study Bernoulli shifts in Section 6 and make the following observation.

\begin{prop} (Proposition \ref{prop:bernoulli-shift})
Let $A$ be a unital simple $C^*$-algebra with a unique trace. Suppose $A$ is not isomorphic to the algebra of complex numbers $\C$. Then the Bernoulli action $\sigma:\Z^m\to \Aut(\bigotimes_{\Z^m} A)$ is strongly outer. 
\end{prop}
\noindent In addition to this proposition, we observe that the Rokhlin dimension of a Bernoulli shift can be bounded by the Rokhlin dimension of a subshift (Proposition \ref{prop:subalgebra-trick}). Combining these two results, we obtain a large class of examples of $\Z^m$-actions with finite Rokhlin dimension. 

Let us describe how the paper is organized. In Section 2 we review the definition of Rokhlin dimension for $\Z^m$-actions. We show that as in the case of automorphisms, strong outerness is a necessary condition for an $\Z^m$-action to have finite Rokhlin dimension. We also review the notion of continuous $W^*$-bundles and discuss group actions on them. In Section 3 we prove a Rokhlin type theorem for $\Z^m$-actions on continuous $W^*$-bundles. The ideas and techniques are almost identical to the case of automorphisms, hence we will only outline the important steps and we refer the reader to \cite{Lia16} for full details. Section 4 is devoted to the technical notion we call the ``approximate Rokhlin property'', which already appeared in our previous paper. Here it serves the same purpose: a bridge connecting the Rokhlin property at the level of $W^*$-bundles and Rokhlin dimension at the level of $C^*$-algebras. Section 5 contains a simple but crucial observation (Proposition \ref{prop:Zm-Rokhlin}), which allows us to obtain finite Rokhlin dimension for $\Z^m$-actions from ``equivariant Rokhlin towers for automorphisms'' This section also contains a technical result which establishes the existence of these Rokhlin towers. Combining the observations and results from previous sections, we establish Theorem \ref{thm:1-1}. In the final section, Section 6, we study Bernoulli shifts on infinite tensor products. Combining the Rokhlin type theorem from Section 5 and an observation involving subshifts, we show that for a large class of $C^*$-algebras, the $\Z^m$-Bernoulli shift has finite Rokhlin dimension.  \\

\noindent\textbf{Acknowledgment:} Much of this work was done under the Pritchard Dissertation Fellowship granted by the Department of Mathematics at Pennsylvania State University. The author wishes to thank his dissertation advisor, Nate Brown, for many valuable discussions, including the suggestion of Corollary \ref{lem:bernoulli}.  The author also wishes to thank an anonymous referee for many helpful comments and the suggestion of Proposition \ref{prop:subalgebra-trick}, which streamlines the proof of Corollary \ref{lem:bernoulli}.

\section{Preliminaries}

Let us first fix some notations. Let $A$ be a $C^*$-algebra, $\Gamma$ be a discrete group, and $\alpha:\Gamma\to \Aut(A)$ be a group action. We write $\alpha^g$ for the automorphism $\alpha(g)\in \Aut(A)$. Suppose the trace space $T(A)$ of $A$ is nonempty, then the set of $\alpha^g$-invariant traces will be denoted $T(A)^{\alpha^g}$. Given any $\tau\in T(A)$, we write $(\pi_\tau, \hh_\tau)$ for the GNS representation associated to $\tau$.

\begin{defn} (\cite[Definition 2.7]{MS12b}) Let $A$ be a unital $C^*$-algebra with nonempty trace space $T(A)$, and let $\Gamma$ be a discrete group. An action $\alpha:\Gamma\to \Aut(A)$ is called \emph{strongly outer} if for every $g\in \Gamma\setminus\{e\}$ and every $\tau\in T(A)^{\alpha^g}$, the weak extension $\tilde{\alpha}^g\in \Aut(\pi_\tau(A)'')$ is an outer automorphism.

\end{defn}

From now on we focus on the case $\Gamma = \Z^m$. We use $\xi_1,...,\xi_m$ for the standard generators of $\Z^m$ (i.e., $\xi_i$ is the unit vector with 1 on the i-th place and 0's elsewhere). The automorphisms $\alpha^{\xi_i}$ will be denoted $\alpha^i$ for short.
For each $n\in \N$, we write $B_n^m$ for the ``box'' $\{0,1,...,n-1\}^m$ in $\Z^m$. When there is no possible confusion, we drop the dimension $m$ and simply write $B_n$.

Below we recall the definition of Rokhlin dimension for $\Z^m$-actions. Roughly speaking, an action has finite Rokhlin dimension if we can find families of positive contractions in the algebra which are almost central, almost cyclically permuted by the action, almost mutually orthogonal, and almost sum to 1.

\begin{defn} \label{defn:Zm-Rok} (\cite[Definition 1.6]{Sza15}) Let $A$ be a unital $C^*$-algebra, and let $\alpha:\Z^m\to \Aut(A)$ be a group action. We say that the action $\alpha$ has Rokhlin dimension at most $d$, and write $\rdim(\alpha) \leq d$, if for any finite subset $F\subseteq A$, $\e > 0$, and $n\in \N$, there exist positive contractions 
	$$
	(f_v^{(\ell)})_{v\in B_n}^{\ell=0,1,...,d}
	$$
	in $A$ satisfying the following properties:
	\begin{enumerate}
		\item[(1)] $\| \alpha^w(f_v^{(\ell)}) - f^{(\ell)}_{(v+w \mod n\Z^m)}\| < \e \;\;\;\;\;\; (0\leq \ell\leq d,\; v\in B_n,\; w\in \Z^m)$.
		\item[(2)] $\| f^{(\ell)}_v f^{(\ell)}_{v'} \| < \e\;\;\;\;\;\; (0\leq \ell \leq d, \; v,v'\in B_n,\; v\neq v')$.
		\item[(3)] $\| [f_v^{(\ell)}, a  ] \| < \e\;\;\;\;\;\; (0\leq \ell\leq d,\; v\in B_n,\; a\in F)$.
		\item[(4)] $\left\|  \sum_{\ell=0}^d\sum_{v\in B_n}f_v^{(\ell)} - 1_A \right\| < \e$.
	\end{enumerate}
\end{defn}

As in the case of $\Z$-actions, strong outerness is a necessary condition for an action to have finite Rokhlin dimension (when the $C^*$-algebra has traces).

\begin{prop} \label{prop:main-converse} Let $A$ be a unital $C^*$-algebra with nonempty trace space $T(A)$, and let $\alpha:\Z^m\to \Aut(A)$ be a group action. If $\rdim(\alpha) = d < \infty$, then $\alpha$ is strongly outer.
\end{prop}
\begin{proof}
	Since the argument is almost identical to the one for implication (1) $\implies$ (3) of \cite[Corollary 1.2]{Lia16}, we only sketch the proof here. Given $v\in \Z^m$, let $u$ be the implementing unitary in the crossed product $A\rtimes_{\alpha^v}\Z$. Given any tracial state $\varphi$ on $A\rtimes_{\alpha^v}\Z$,  one can show that $\varphi(au^n) = 0$ for any $a\in A$ and $n\in \N$ by approximating $1_A$ by the sum of the Rokhlin elements. Then the proof of \cite[Lemma 4.4]{Kis96} shows that a weak extension of $\alpha^v$ cannot be inner.
\end{proof}

\begin{rem} \label{rem:trace}
Note that the argument above also shows that for each $v\in \Z^m\setminus \{0\}$, the restriction map $r:T(A\rtimes_{\alpha^v}\Z)\to T(A)^{\alpha^v}$ is injective (hence bijective).
\end{rem}

We briefly review some basic definitions of continuous $W^*$-bundles, which was introduced by Ozawa in \cite{Oza13} (see also \cite{BBSTWW15}). A \emph{continuous $W^*$-bundle} consists of the following data:
\begin{itemize}
	\item a $C^*$-algebra $\hm$,
	\item a compact metrizable space $K$ such that $C(K)$ sits inside the center of $\hm$, and
	\item a faithful conditional expectation $E:\hm\to C(K)$ satisfying the tracial condition $E(ab) = E(ba)$ for all $a,b\in \hm$,
\end{itemize} 
such that the norm-closed unit ball of $\hm$ is complete in the \emph{uniform 2-norm} $\|\cdot\|_{2,u}$, defined by
$$
\|x\|_{2,u} := \| E(x^*x) \|^{1/2}. 
$$
For each $\lambda\in K$, the GNS representation $\pi_\lambda(\hm)$ coming from the tracial state $\tau_\lambda := ev_\lambda\circ E$ is called a \emph{fiber} of the bundle $(\hm, K, E)$.

Given a unital simple separable $C^*$-algebra with nonempty trace space $T(A)$, we define the \emph{uniform 2-norm} $\|\cdot\|_{2,u}$ on $A$ by
$$
\|a\|_{2,u} := \sup_{\tau\in T(A)}\tau(a^*a)^{1/2}.
$$
Let $\overline{A}^u$ be the completion of $A$ with respect to $\|\cdot\|_{2,u}$. Ozawa proved in \cite{Oza13} that when $T(A)$ is a \emph{Bauer simplex}, i.e., when the extreme boundary $\partial_e T(A)$ is compact, the set of continuous functions on $\partial_e T(A)$ can be embedded into the center of $\overline{A}^u$. Moreover, there is a faithful tracial conditional expectation
$$
E:\overline{A}^u \to C(\partial_e T(A))
$$
satisfying $E(a)(\tau) = \tau(a)$ for each $a\in A$ and $\tau\in \partial_e T(A)$. Therefore $(\overline{A}^u, \partial_e T(A), E)$ has the structure of a continuous $W^*$-bundle. In this case each fiber $\pi_\tau(\overline{A}^u)$ is isomorphic to the weak closure $\pi_\tau(A)''$.

Given a free ultrafilter $\omega$ on $\N$, one can form the \emph{ultrapower} $(\hm^\omega, E^\omega, K^\omega)$ and the \emph{central sequence $W^*$-bundle} $(\hm^\omega\cap \hm', E^\omega, K^\omega)$ of a continuous $W^*$-bundle $(\hm, K, E)$ in a way similar to the case of tracial von Neumann algebras. The uniform 2-norm on $\hm^\omega$ will be denoted $\|\cdot\|_{2,u}^{(\omega)}$. For details, we refer the reader to \cite{BBSTWW15} and \cite{Lia16}.

Now we discuss group actions on continuous $W^*$-bundles. Let $(\hm, K, E)$ be a continuous $W^*$-bundle and let $\Gamma$ be a discrete group. A \emph{$\Gamma$-action} on $(\hm, K, E)$ consists of a pair $(\beta, \sigma)$, where
$$
\beta:\Gamma\to \Aut(\hm)\;\;\;\;\;\; \text{ and }\;\;\;\;\;\; \sigma:\Gamma \to \text{Homeo}(K)
$$
are group homomorphisms, such that the diagram
\begin{center}
\begin{tikzpicture}[node distance=2cm, auto]
  \node (1) {$\hm$};
  \node (2) [right of=1] {$\hm$};
  \node (6) [below of=1] {$C(K)$};
  \node (7) [below of=2] {$C(K)$};
  \draw[->] (1) to node {$\beta^g$} (2);
  \draw[->] (6) to node {$\bar{\sigma}^g$} (7);
  \draw[->] (2) to node {$E$} (7);
  \draw[->] [swap] (1) to node {$E$} (6);
\end{tikzpicture}.
\end{center}
commutes for every $g\in \Gamma$ (here $\bar{\sigma}:\Gamma\to \Aut(C(K))$ is the action induced by $\sigma$). Given a group action $(\beta, \sigma):\Gamma\to \Aut(\hm, K, E)$, for each $g\in \Gamma$ and $\lambda\in K$, there is a well-defined *-isomorphism
\begin{align*}
\beta_\lambda^g: \pi_\lambda(\hm) &\to \pi_{(\sigma_g)^{-1}(\lambda)}(\hm),\\
\pi_\lambda(x)&\mapsto  \pi_{(\sigma_g)^{-1}(\lambda)}(\beta^g(x)).
\end{align*}
When $\sigma$ is the trivial action on $K$, this gives rise to a group action
$$
\beta_\lambda:\Gamma\to \Aut(\pi_\lambda(\hm)).
$$

Let $\omega$ be a free ultrafilter on $\N$. A group action $(\beta, \sigma)$ can be promoted to an action on the ultrapower or the central sequence $W^*$-bundle. More precisely, for each $g\in \Gamma$ we have automorphisms
\begin{align*}
\beta_\omega^g:\hm^\omega&\to \hm^\omega,\\
(x_n)_n&\mapsto (\beta^g(x_n))_n,
\end{align*}
and
\begin{align*}
\bar{\sigma}_\omega^g:C(K^\omega)&\to C(K^\omega),\\
(f_n)_n&\mapsto (f_n\circ \sigma_g)_n
\end{align*}
(recall that we can identify $C(K^\omega)$ with the $C^*$-ultrapower $C(K)_\omega$). Therefore we obtain actions
$$
\beta_\omega:\Gamma\to \Aut(\hm^{\omega})\;\;\;\;\;\; \text{ and }\;\;\;\;\;\; \bar{\sigma}_\omega:\Gamma \to \Aut(C(K^\omega)).
$$
One easily checks that the diagram
\begin{center}
\begin{tikzpicture}[node distance=2cm, auto]
  \node (1) {$\hm^\omega$};
  \node (2) [right of=1, node distance=2.5cm] {$\hm^\omega$};
  \node (6) [below of=1] {$C(K^\omega)$};
  \node (7) [below of=2] {$C(K^\omega)$};
  \draw[->] (1) to node {$\beta_\omega^g$} (2);
  \draw[->] (6) to node {$\bar{\sigma}_\omega^g$} (7);
  \draw[->] (2) to node {$E^\omega$} (7);
  \draw[->] [swap] (1) to node {$E^\omega$} (6);
\end{tikzpicture}.
\end{center}
commutes for every $g\in \Gamma$. We write $(\hm^\omega\cap \hm')^{\beta_\omega}$ for the fixed point subalgebra of the central sequence $W^*$-bundle $\hm^\omega\cap \hm'$.

\section{$W^*$-Rokhlin property}

The goal of this section is to prove a Rokhlin type theorem for $\Z^m$-actions on continuous $W^*$-bundles. Since the argument is almost the same as \cite[Theorem 4.1]{Lia16}, we will be terse for most of the time. Our first objective is to embed arbitrarily large matrices into the fixed point algebra of the central sequence $W^*$-bundle, as in the following proposition and corollary. Recall that a completely positive (c.p.) map $\varphi:A\to B$ between $C^*$-algebras is called \emph{order zero} if $\varphi(a)\varphi(b) = 0$ for all positive elements $a,b\in A$ satisfying $ab=0$.

\begin{prop} \label{prop:fixedpointsubalgebra} (cf. \cite[Proposition 3.4]{Lia16}) Let $(\hm, K, E)$ be a separable continuous $W^*$-bundle, meaning that $\hm$ contains a countable $\|\cdot\|_{2,u}$-dense subset, and $(\beta,\sigma):\Z^m\to \Aut(\hm, K, E)$ be a group action. Suppose
\begin{enumerate}
	\item[(1)] the action is trivial on $K$, i.e., $\sigma = \id_K$;
	\item[(2)] $\dim(K) = d < \infty$;
	\item[(3)] each fiber $\pi_\lambda(\hm)$ is isomorphic to the hyperfinite II$_1$ factor $\hr$;
	\item[(4)] the fiber action $\beta:\Z^m\to \Aut(\pi_\lambda(\hm))$ is pointwise outer for each $\lambda\in \K$.
\end{enumerate}
Then for each $p\in \N$, there exist completely positive contractive (c.p.c.) order zero maps $\psi_1,...,\psi_{d+1}:M_p(\C)\to (\hm^\omega\cap \hm')^{\beta_\omega}$ with commuting ranges such that $\psi_1(1)+\cdots + \psi_{m+1}(1) = 1_\hm$.
\end{prop}

\begin{cor} (cf. \cite[Proposition 3.3]{Lia16}) Let $(\hm, K, E)$ be a separable continuous $W^*$-bundle and $(\beta,\sigma):\Z^m\to \Aut(\hm, K, E)$ be a group action. Suppose
	\begin{enumerate}
		\item[(1)] the action is trivial on $K$, i.e., $\sigma = \id_K$;
		\item[(2)] $\dim(K) = d < \infty$;
		\item[(3)] each fiber $\pi_\lambda(\hm)$ is isomorphic to the hyperfinite II$_1$ factor $\hr$;
		\item[(4)] the fiber action $\beta:\Z^m\to \Aut(\pi_\lambda(\hm))$ is pointwise outer for each $\lambda\in \K$.
	\end{enumerate}
Then for every $p\in \N$, there is a unital *-homomorphism $\sigma:M_p(\C)\to  (\hm^\omega\cap \hm')^{\beta_\omega}$.
\end{cor}
\begin{proof} This follows immediately from Proposition \ref{prop:fixedpointsubalgebra} and \cite[Lemma 7.6]{KR14}.
\end{proof}

To prove Proposition \ref{prop:fixedpointsubalgebra}, we make use of the finite dimensionality of $K$ and build a collection of c.p.c. maps from a matrix algebra $M_p(\C)$ into $\hm$ which are almost central, almost invariant under the action $\beta$, and have almost commuting ranges. The following lemma is simply a $\Z^m$-version of \cite[Lemma 3.8]{Lia16} (with identical proof).

\begin{lem}
	\label{lem:fixpointcpc}
	Under the same assumptions as Proposition \ref{prop:fixedpointsubalgebra}, for every $\e > 0, p\in \N$, and norm bounded $\|\cdot\|_{2,u}$-compact subset $\Omega\subseteq \hm$, there exist c.p.c. maps $\psi^{(1)},...,\psi^{(d+1)}:M_p(\C)\to \hm$ (remember that $d = \dim(K)$) such that
	\begin{enumerate}
		\item[(i)] $\|[\psi^{(\ell)}(e),x]\|_{2,u} < \e$\;\;\;\;\;\; $(1\leq \ell\leq d+1,\; e\in M_p(\C)_1,\; x\in \Omega)$;
		\item[(ii)] $\| [\psi^{(\ell)}(e), \psi^{(k)}(f) ] \|_{2,u} < \e$\;\;\;\;\;\; $(1\leq \ell\neq k \leq d+1,\;e, f\in M_p(\C)_1)$;
		\item[(iii)] $\| \psi^{(\ell)}(1)\psi^{(\ell)}(e^*e) - \psi^{(\ell)}(e)^*\psi^{(\ell)}(e) \|_{2,u} < \e$\;\;\;\;\;\; $(1\leq \ell\leq d+1,\; e\in M_p(\C)_1)$;
		\item[(iv)] $\sum_{\ell=1}^{d+1}\psi^{(\ell)}(1)  = 1_\hm$;
		\item[(v)] $\| \beta^i(\psi^{(\ell)}(e))-\psi^{(\ell)}(e)\|_{2,u} < \e$\;\;\;\;\;\; $(1\leq i\leq m,\; 1\leq \ell\leq d+1,\; e\in M_p(\C)_1)$.
	\end{enumerate}
\end{lem}

Now the proof of Proposition \ref{prop:fixedpointsubalgebra} can be completed by mapping sequences of c.p.c. maps as found in Lemma \ref{lem:fixpointcpc} into the ultrapower:

\begin{proof} (of Proposition \ref{prop:fixedpointsubalgebra})
	Let $\{x_1,x_2,...\}$ be a countable $\|\cdot\|_{2,u}$-dense subset of $\hm$. By Lemma \ref{lem:fixpointcpc}, for each $n\in \N$ there exist c.p.c. maps $\psi_n^{(1)}, \psi_n^{(2)},...,\psi_n^{(d+1)}:M_p(\C)\to \hm$ such that
	(ii) to (v) in Lemma \ref{lem:fixpointcpc} hold with $\e = \frac{1}{n}$ and
	\begin{enumerate}
		\item[(i')] $\| [\psi_n^{(\ell)}(e), x_j] \|_{2,u} < \frac{1}{n}\;\;\;\;\;\; (1\leq \ell\leq d+1,\; e\in M_p(\C)_1,\; 1\leq j\leq n)$.
	\end{enumerate}
	For each $1\leq \ell\leq d+1$ define c.p.c. maps $\psi_\ell:M_p(\C)\to \hm^{\omega}$ by
	$$
	\psi_\ell(e) = \pi_\omega(\psi_1^{(\ell)}(e), \psi_2^{(\ell)}(e), \psi_3^{(\ell)}(e),... ),
	$$
	where $\pi_\omega$ is the quotient map from $\ell^\infty(\N,\hm)$ onto $\hm^\omega$. Then by construction 
	\begin{itemize}
		\item the image of $\psi_\ell$ belongs to $(\hm^\omega\cap \hm')^{\beta_\omega}$;
		\item $\psi_1, \psi_2,...,\psi_{d+1}$ have commuting ranges;
		\item $\psi_1(1)+\cdots + \psi_{d+1}(1) = 1_\hm$;
		\item $\psi_\ell(1)\psi_{\ell}(e^*e) = \psi_\ell(e)^*\psi_\ell(e)$\;\;\;\;\;\; $(1\leq \ell\leq d+1,\; e\in M_p(\C))$.
	\end{itemize}
	It remains to show that each $\psi_\ell$ is order zero, but this follows from the fourth bullet (see \cite[Lemma 3.9]{Lia16}).
\end{proof}

Now we define the $W^*$-Rokhlin property and state the Rokhlin type theorem.

\begin{defn}  \label{defn:W-Rok} Let $(\hm, K, E)$ be a separable continuous $W^*$-bundle and $(\beta,\sigma):\Z^m\to \Aut(\hm, K, E)$ be a group action. We say $(\beta, \sigma)$ has the \emph{$W^*$-Rokhlin property} if for every $n\in \N$ there exist projections $\{p_v\}_{v\in B_n}$ in $\hm^\omega\cap \hm'$ such that
	\begin{enumerate}
		\item[(1)] $(\beta_\omega)^u(p_v) = p_{(v+u \mod n\Z^m)}  \;\;\;\;\;\; (v\in B_n,\; u\in \Z^m)$.
		\item[(2)] $\sum_{v\in B_n} p_v = 1_\hm$.
		\item[(3)] $\| 1_\hm-p_0 \|_{2,u}^{(\omega)} < 1$.
	\end{enumerate}
\end{defn}

\begin{thm} \label{thm:W-Rokhlin} Let $(\hm, K, E)$ be a separable continuous $W^*$-bundle and $(\beta,\sigma):\Z^m\to \Aut(\hm, K, E)$ be a group action. Suppose
	\begin{enumerate}
		\item[(1)] the action is trivial on $K$, i.e., $\sigma = \id_K$;
		\item[(2)] $\dim(K) = d < \infty$;
		\item[(3)] each fiber $\pi_\lambda(\hm)$ is isomorphic to the hyperfinite II$_1$ factor $\hr$;
		\item[(4)] the fiber action $\beta:\Z^m\to \Aut(\pi_\lambda(\hm))$ is pointwise outer for each $\lambda\in K$.
	\end{enumerate}
	Then $(\beta, \sigma)$ has the $W^*$-Rokhlin property.
\end{thm}

The idea of proving Theorem \ref{thm:W-Rokhlin} is the same as \cite[Theorem 4.1]{Lia16}. We first obtain a partition of unity which are almost mutually orthogonal in terms of the uniform 2-norm, and then glue the Rokhlin tower in each fiber (the fiberwise result was due to Ocneanu \cite{Ocn81}) to a global section. The following lemma provides the partition of unity we need.

\begin{lem} Let $(\hm, K, E)$ be a separable continuous $W^*$-bundle and $(\beta, \sigma):\Z^m\to \Aut(\hm,K,E)$ be a group action. Suppose
	\begin{enumerate}
		\item[(1)] the action is trivial on $K$, i.e., $\sigma = \id_K$;
		\item[(2)] for each $p\in \N$ there exists a unital *-homomorphism $\sigma:M_p(\C)\to (\hm^\omega \cap \hm')^{\beta_\omega}$.
	\end{enumerate}
Then given a norm bounded $\|\cdot\|_{2,u}$-compact subset $\Omega$ of $\hm$, $\e > 0$, and a partition of unity $\{g_j\}_{j=1}^N$ subordinate to some open cover of $K$, there exist positive contractions $a_1,a_2,...,a_N$ in $\hm$ such that
\begin{enumerate}
	\item[(i)] $\| [a_j,x] \|_{2,u} < \e\;\;\;\;\;\; (1\leq j\leq N,\; x\in \Omega)$;
	\item[(ii)] $\| (\beta_\omega)^i(a_j) - a_j \|_{2,u} < \e\;\;\;\;\;\; (1\leq i\leq m,\; 1\leq j\leq N)$.
	\item[(iii)] $\| a_ia_j\|_{2,u} < \e\;\;\;\;\;\; (1\leq i\neq j\leq N)$;
	\item[(iv)] $\| [a_j^{1/2}, x] \|_{2,u} < \e\;\;\;\;\;\;\; (1\leq j\leq N,\; x\in \Omega)$;
	\item[(v)] $\| a_i^{1/2}a_j^{1/2} \|_{2,u} < \e\;\;\;\;\;\;\; (1\leq i\neq j \leq N)$;
	\item[(vi)] $\left| \tau_\lambda(a_jx) - g_j(\lambda)\tau_\lambda(x) \right| < \e\;\;\;\;\;\; (1\leq j\leq N,\; x\in \Omega,\; \lambda\in K)$;
	\item[(vii)] $\sum_{j=1}^N a_j = 1_\hm$.
\end{enumerate}
\end{lem}
\begin{proof}
	The proof is identical to \cite[Lemma 4.4]{Lia16}.
\end{proof}

\begin{proof} (of Theorem \ref{thm:W-Rokhlin}) The only difference with the proof of \cite[Theorem 4.1]{Lia16} is that in order to obtain a Rokhlin tower in each fiber, we invoke Ocneanu's Rokhlin type theorem \cite[Theorem 2]{Ocn81} instead of Connes' Rokhlin lemma for automorphisms.
\end{proof}


\section{Approximate Rokhlin Property}

In this section we discuss the notion of ``equivariant $\sigma$-ideals'' for actions of discrete amenable groups. For $\Z$-actions, this was introduced by the author in \cite[Definition 5.5]{Lia16}. Later it was generalized to arbitrary second-countable, locally compact groups by Szab\'o in \cite{Sza17}. Here we review the definition and the fact that the trace kernel ideal is an equivariant $\sigma$-ideal in the ultrapower (see Lemma \ref{lem:alpha-sigma-ideal}). A proof is included mainly for the reader's convience. For a more general discussion see \cite[Section 3]{Sza17}.

The notion of $\sigma$-ideals was first introduced by Kirchberg and R\o rdam in \cite{KR14}. It was designed to capture the behavior of a quasicentral approximate unit. When $B$ is a $C^*$-algebra, $\Gamma$ is a discrete countable amenable group, and $\alpha:\Gamma\to \Aut(B)$ is a group action, the following definition essentially packages an almost invariant quasicentral approximate unit.

\begin{defn} (cf. \cite[Definition 3.1]{Sza17})
Let $B$ be a $C^*$-algebra, $\Gamma$ a discrete countable amenable group, and $\alpha:\Gamma\to \Aut(B)$ a group action. Let $J$ be a closed ideal in $B$ such that $\alpha^g(J) = J$ for every $g\in \Gamma$. We say $J$ is an \textit{$\Gamma$-$\sigma$-ideal} if for every separable $C^*$-subalgebra $C$ of $B$ satisfying $\alpha^g(C)=C$ for every $g\in \Gamma$, there exists a positive contraction $u\in C'\cap J$ such that
\begin{enumerate}
\item[(1)] $\alpha^g(u) = u$ for every $g\in \Gamma$.
\item[(2)] $uc = c$ for every $c\in C\cap J$.
\end{enumerate}
\end{defn}

Let $A$ be a unital separable $C^*$-algebra and let $A_\omega$ be the ultrapower of $A$ (here we fix once and for all a free ultrafilter $\omega$ on $\N$). Suppose that $A$ has nonemepty trace space $T(A)$. Given an element $a$ in $A_\omega$ represented by a sequence $(a_1,a_2,...)$, recall from \cite{KR14} that the seminorm $\|\cdot\|_{2,\omega}$ on $A_\omega$ is defined by
$$
\|a\|_{2,\omega} := \lim_{n\to \omega} \sup_{\tau\in T(A)} \tau(a_n^*a_n)^{1/2}.
$$
The \emph{trace kernel ideal} $J_A$ is defined to be
$$
J_A := \{ e\in A_\omega : \|e\|_{2,\omega} = 0 \}.
$$

\begin{lem} \label{lem:alpha-sigma-ideal}
Let $A$ be a unital separable $C^*$-algebra with $T(A)\neq \emptyset$, $\Gamma$ a discrete countable amenable group, and $\alpha:\Gamma\to \Aut(A)$ a group action. Then the trace-kernel ideal $J_A$ is an $\Gamma$-$\sigma$-ideal in $A_\omega$.
\end{lem}
\begin{proof}
First note that $(\alpha_\omega)^g(J_A) = J_A$ for every $g\in \Gamma$ since $(\alpha_\omega)^g$ preserves the uniform 2-norm. By \cite[Proposition 4.6]{KR14} $J_A$ is a $\sigma$-ideal in $A_\omega$, so there exists a positive contraction $e\in C'\cap J_A$ such that $ec = c$ for every $c\in C\cap J_A$. Let $F_1\subseteq F_2\subseteq \cdots \subseteq \Gamma$ be an increasing sequence of finite subsets whose union is $\Gamma$. By amenability for each $n\in \N$ we can find a $(F_n, 1/n)$-invariant (finite) subset $K_n$ of $\Gamma$, in the sense that
$$
\max_{s\in F_n} \frac{ \left|sK_n\triangle K_n\right|}{|K_n|} < \frac{1}{n}.
$$
For each $n\in \N$, define 
$$
e^{(n)} := \frac{1}{|K_n|}\sum_{g\in K_n} (\alpha_\omega)^g(e).
$$
One checks that $[e^{(n)}, c] = 0$ for every $c\in C$ and $e^{(n)}c = c$ for every $c\in C\cap J$. Moreover, for each $s\in F_n$ we have
\begin{align*}
\| (\alpha_\omega)^s(e^{(n)}) - e^{(n)}\| &= \frac{1}{|K_n|}\left\| \sum_{g\in K_n}(\alpha_\omega)^{sg}(e) - \sum_{g\in K_n}(\alpha_\omega)^g(e) \right\| \\
&\leq \frac{1}{|K_n|}|sK_n\triangle K_n| < \frac{1}{n}.
\end{align*}

To finish the proof, we invoke Kirchberg's $\e$-test \cite[Lemma 3.1]{KR14}. Let $d = (d_1,d_2,...)$ be a strictly positive contraction in the separable $C^*$-algebra $C\cap J_A$, and $\{g_k\}_{k=1}^\infty$ be a list of elements in $\Gamma$. Take a dense sequence $\{c^{(k)}\}_{k=1}^\infty$ in the unit ball of $C$, and represent each $c^{(k)}$ by $(c_1^{(k)},c_2^{(k)},...)$. Let each $X_n$ be the set of positive contractions in $A$, and define functions $f_n^{(k)}:X_n\to [0,\infty)$ by
\begin{align*}
f_n^{(1)}(x) &= \| (1-x)d_n\|,\\
f_n^{(2)}(x) &= \|x\|_{2,u},\\
f_n^{(2k+1)} &= \|\alpha^{g_k}(x)-x\|\;\;\;\;\;\; (k\in \N),\\
f_n^{(2k+2)} &= \|c_n^{(k)}x-xc_n^{(k)}\|\;\;\;\;\;\;\;(k\in \N).
\end{align*}
Given $m\in \N$ and $\e>0$, there exists $\ell\in \N$ sufficiently large (along the filter $\omega$) so that
$$
e^{(\ell)} = (e_1^{(\ell)},e_2^{(\ell)},...)
$$
in $A_\omega$ satisfies
$$
f_\omega^{(k)}(e_1^{(\ell)},e_2^{(\ell)},...) < \e\;\;\;\;\;\; (1\leq k\leq m).
$$
Now the proof is finished by applying the $\e$-test.
\end{proof}

As in \cite{Lia16}, we define a technical property called the \emph{approximate Rokhlin property}, which would serve as a bridge connecting the $W^*$-Rokhlin property and the Rokhlin dimension. Recall that the relative commutant $A_\omega \cap A'$ is called the \emph{central sequence algebra} of $A$. Following Kirchberg \cite{Kir06}, we will denote the central sequence algebra of $A$ by $F(A)$.

\begin{defn} Let $A$ be a unital $C^*$-algebra and $\alpha:\Z^m\to \Aut(A)$ be a group action. We say $\alpha$ has the \emph{approximate Rokhlin property} if for every $p\in \N$ there exist positive contractions $\{f_{v}\}_{v\in B_p}$ in $F(A)$ such that
	\begin{itemize}
		\item[(1)] $(\alpha_\omega)^u(f_v) = f_{(v+u \mod p\Z^m)} \;\;\;\;\;\; (v\in B_p,\;u\in \Z^m)$.
		\item[(2)] $f_vf_u = 0\;\;\;\;\;\; (v,u\in B_p,\; v\neq u)$.
		\item[(3)] $1_A - \sum_{v\in B_p}f_v$ belongs to $J_A\cap F(A)$.
		\item[(4)] $\sup_{k\in \N} \| 1_A - f_0^k \|_{2,\omega} < 1$.
	\end{itemize}
(here 0 is the origin in $\Z^m$).
\end{defn}

With the aid of Lemma \ref{lem:alpha-sigma-ideal}, it is not hard to pass from the $W^*$-Rokhlin property to the approximate Rokhlin property.

\begin{lem} \label{lem:W-Rok-to-Approx-Rok} (cf. \cite[Propoition 5.8]{Lia16})
Let $A$ be a unital $C^*$-algebra and $\alpha:\Z^m\to \Aut(A)$ be a group action. Suppose the induced action $(\tilde{\alpha}, \bar{\alpha}):\Z^m\to \Aut(\overline{A}^u,\partial_e T(A),E)$ has the $W^*$-Rokhlin property. Then $\alpha$ has approximate Rokhlin property.
\end{lem}
\begin{proof}
To simplify the notion, we write $(\hm,K, E)$ for the triple $(\overline{A}^u, \partial_e T(A), E)$. Let $\{p_v\}_{v\in B_p}$ be projections in $\hm^\omega\cap \hm'$ satisfying the $W^*$-Rokhlin property. Recall that the canonical map $F(A)\to \hm^\omega\cap \hm'$ is surjective (\cite{KR14}). Lifting the projections $\{p_v\}$ along this surjection, we obtain positive contractions $\{f_v'\}$ in $F(A)$ such that
\begin{itemize}
\item[(1)] $(\alpha_\omega)^u(f_v') - f'_{(v+u \mod p\Z^m)}\in J_A\cap F(A) \;\;\;\;\;\;\; (v\in B_p,\; u\in \Z^m)$,
\item[(2)] $f'_vf'_u\in J_A\cap F(A)\;\;\;\;\;\; (u,v\in B_p,\; u\neq v)$, 
\item[(3)] $1_A - \sum_v f'_v\in J_A\cap F(A)$.
\end{itemize}
Put $C := C^*( A, \{ (\alpha_\omega)^u(f_v'): u\in \Z^m,\; v\in B_p \} )$. Then $C$ is a separable $C^*$-subalgebra of $A_\omega$ satisfying $(\alpha_\omega)^u(C) =C$ for every $u\in \Z^m$. Since $J_A$ is an $\Z^m$-$\sigma$-ideal, there exists a positive contraction $u\in C'\cap J_A$ such that $(\alpha_\omega)^w(u) = u$ for all $w\in \Z^m$ and $uc = c$ for all $c\in C\cap J_A$. Define 
$$
f_v := (1-u)f_v'(1-u)\;\;\;\;\;\; (v\in B_{p}).
$$
One checks that the family $\{f_v\}_{v\in B_p}$ has the desired properties.
\end{proof}

\begin{rem}
	Here we use $F(A)$ mainly as a shorthand notation for the central sequence $A_\omega\cap A'$. In the literature $F(A)$ was also defined for non-unital $C^*$-algebra (see \cite{Kir06}). One may ask whether a version of Lemma \ref{lem:W-Rok-to-Approx-Rok} also holds in the non-unital case. Note that this requires generalizing the results on surjectivity of $A_\omega\cap A'\to \mathcal{M}^\omega \cap \mathcal{M}'$ and on the trace kernel ideals (see \cite{KR14}).
\end{rem}




\section{A Rokhlin type theorem}

In this section we establish our Rokhlin type theorem for $\Z^m$-actions on $C^*$-algebras (Theorem \ref{thm:main}). Let $\alpha:\Z^m\to \Aut(A)$ be a group action. The key observation, motivated by \cite{Nak99}, is that suppose for every canonical generator $\alpha^i$, we can find Rokhlin towers for $\alpha^i$ which are almost fixed by all the other generators $\alpha^1,...,\alpha^{i-1},\alpha^{i+1},...,\alpha^m$, then we can simply ``multiply'' the towers together and get Rokhlin towers for $\alpha$. Here is the precise statement:

\begin{prop} \label{prop:Zm-Rokhlin} Let $A$ be a unital $C^*$-algebra and $\alpha:\Z^m\to \Aut(A)$ be a group action. Suppose there exists $d\in \N$ such that for each $i\in \{1,...,m\}$ the following holds: for every $\e > 0$, $p\in \N$, finite subset $F\subseteq A$, there exist positive contractions
	$$
	\{g_{i,k}^{(\ell)}: 0\leq k\leq p-1, \;0\leq \ell\leq d\}
	$$
	in $A$ such that
	\begin{itemize}
		\item[(1)] $\|\alpha^i(g_{i,k}^{(\ell)}) - g_{i,(k+1 \mod p)}^{(\ell)} \| < \e \;\;\;\;\;\; (0\leq k\leq p-1,\; 0\leq \ell\leq d)$,
		\item[(2)] $\| \alpha^j(g_{i,k}^{(\ell)}) - g_{i,k}^{(\ell)} \| < \e\;\;\;\;\;\; (1\leq j\leq m, \; j\neq i,\; 0\leq k\leq p-1,\; 0\leq \ell\leq d)$,
		\item[(3)] $\| [ g_{i,k}^{(\ell)}, x] \| < \e\;\;\;\;\;\; (0\leq k\leq p-1,\; 0\leq \ell\leq d,\; x\in F)$,
		\item[(4)] $\| g_{i,k}^{(\ell)} g_{i,n}^{(\ell)} \| < \e\;\;\;\;\;\; (0\leq k\neq n\leq p-1,\; 0\leq \ell\leq d)$,
		\item[(5)] $\left\| \sum_{k,\ell}g_{i,k}^{(\ell)} - 1_A \right\| < \e$.
	\end{itemize}
	Then $\rdim(\alpha) \leq (d+1)^m-1$.
\end{prop}
\begin{proof}
	Fix $\e > 0$, a finite subset $F\subseteq A$, and $p\in \N$. Let 
	$$
	\{g_{1,k}^{(\ell)}: 0\leq k\leq p-1,\; 0\leq \ell\leq d \}
	$$
	be positive contractions in $A$ satisfying (1) to (5) with respect to $\e' > 0$ and $F_1 := F$, where $\e'$ is a small number to be determined later.
	
	Define, for $i = 1,2,...,m-1$, inductively $F_{i+1} := F_i \cup \{g_{i,k}^{(\ell)} \}$ and let $\{ g_{i+1,k}^{(\ell)}:0\leq k\leq p-1,\; 0\leq \ell \leq d \}$ be positive contractions satisfying conditions (1) to (5) with respect to $\e'$ and $F_i$. For $v = (v_1,...,v_m)\in B_p$ and $(\ell_1,...,\ell_m)\in \{0,...,d \}^m$, define (positive contractions)
	$$
	f_v^{(\ell_1,...,\ell_m)} := \left( g_{m,v_m}^{(\ell_m)}\right)^{\frac{1}{2}}\cdots \left( g_{2,v_2}^{(\ell_2)}\right)^{\frac{1}{2}} \left( g_{1,v_1}^{(\ell_1)} \right) \left( g_{2,v_2}^{(\ell_2)}\right)^{\frac{1}{2}}\cdots \left( g_{m,v_m}^{(\ell_m)}\right)^{\frac{1}{2}}.
	$$
	One checks that these elements form $(d+1)^m$ Rokhlin towers for $\alpha$ with respect to $\e$ and $F$, provided that $\e'$ is sufficiently small.
\end{proof}

For the rest of the section we will focus on establishing the conditions described in Proposition \ref{prop:Zm-Rokhlin}. One technical notion we need is the so-called \emph{Property (SI)}; it was first introduced by Sato in \cite{Sat10}, and later reformulated by Kirchberg and R\o rdam in \cite{KR14}.

\begin{defn} (\cite[Definition 2.6]{KR14}) \label{defn:properptySI}  Let $A$ be a unital separable $C^*$-algebra with nonempty trace space. We say $A$ has \emph{property (SI)} if for all positive contractions $e,f\in F(A)$ with $e\in J_A$ and $\sup_{m\in \N}\|1_A-f^m\|_{2,\omega} < 1$, there exists an element $s\in F(A)$ with $fs = s$ and $s^*s = e$.
\end{defn}

\begin{lem} Let $A$ be a unital separable $C^*$-algebra with $T(A)\neq \emptyset$, and $\alpha:\Z^m\to \Aut(A)$ be a group action. Suppose $\alpha$ has the approximate Rokhlin property and $A$ has property (SI). Then for every $i\in \{1,...,m\}$ and $p\in \N$ there exist $g_0,...,g_{p-1},v$ in $F(A)$ such that
\begin{enumerate}
\item[(i)] $g_0,...,g_{p-1}$ are positive contractions,
\item[(ii)] $g_kg_n = 0\;\;\;\;\;\; (0\leq k\neq n \leq p-1)$,
\item[(iii)] $(\alpha_\omega)^i (g_k) = g_{(k+1, \mod p)} \;\;\;\;\;\; (0\leq k\leq p-1)$,
\item[(iv)] $g_0v = v$,
\item[(v)] $\sum_{k=0}^{p-1}g_k +v^*v = 1_A$,
\item[(vi)] $(\alpha_\omega)^{pi}(v) = v$,
\item[(vii)] $(\alpha_\omega)^j(g_k) = g_k\;\;\;\;\;\; (1\leq j\neq i\leq m)$,
\item[(viii)] $(\alpha_\omega)^j(v) = v\;\;\;\;\;\; (1\leq j\neq i\leq m)$.
\end{enumerate}
\end{lem}

\begin{proof} The main idea is almost contained in \cite[Proposition 4.5]{MS14}. However, since the assumptions are slightly different, we include full details for the reader's convenience. Fix $i\in \{1,...,m\}$ and $p\in \N$. First of all, since we are working with ultraproducts, it suffices to find elements satisfying (i) to (viii) up to an arbitrarily small $\e > 0$. Let $\ell$ be a positive integer such that $\frac{2}{\sqrt{\ell p}} < \e$. By the approximate Rokhlin property, there exist $\{f_v\}_{ v\in B_{\ell p} }$ satisfying
	\begin{enumerate}
		\item[(1)] $(\alpha_\omega)^u(f_v) = f_{(v+u \mod (\ell p) \Z^m)} \;\;\;\;\;\;\; (v\in B_{\ell p},\; u\in \Z^m)$,
		\item[(2)] $f_vf_u = 0 \;\;\;\;\;\; (v,u\in B_{\ell p},\; v\neq u)$,
		\item[(3)] $e : = 1_A - \sum_{v\in B_{\ell p}}f_v \in J_A\cap F(A)$,
		\item[(4)] $\sup_{n}\| 1_A - f_0^n \|_{2,\omega} < 1$.
	\end{enumerate}	
Since $A$ has property (SI), there is $w$ in $F(A)$ such that
$$
w^*w = e,\;\;\;\;\;\; \text{and}\;\;\;\;\;\;\; f_0w = w.
$$	
In what follows, the coordinates of a point $v\in \Z^m$ will be denoted $v_1,v_2,...,v_m$. Define, for $0\leq k\leq \ell p -1$, 
$$
g_k':= \sum_{\{v\in B_{\ell p}: v_i = k \}  } f_v
$$
and
$$
s := \frac{1}{ \sqrt{ (\ell p)^{m-1}}   } \sum_{\{u\in B_{\ell p}: u_i = 0\} } (\alpha_\omega)^u(w).
$$
Then clearly we have
\begin{itemize}
	\item $(\alpha_\omega)^j(g_k') = g_k'$ \;\;\;\;\;\;\; ($0\leq k\leq \ell p - 1,\; 1\leq j\neq i\leq m$);
	\item $(\alpha_\omega)^i(g_k') = g_{(k+1 \mod (\ell p))}'; $ \;\;\;\;\;\;\; ($0\leq k\leq \ell p -1$);
	\item $\sum_{k=0}^{\ell p -1}g_k' + e = 1_A$.
\end{itemize}
We make several more computations. First of all,
\begin{align*}
s^*s &= \frac{1}{ (\ell p)^{m-1}   } \sum_{\{u\in B_{\ell p}:u_i = 0\}} \sum_{\{u'\in B_{\ell p}:u_i' = 0\}} \alpha_\omega^{u}(w^*)\alpha_\omega^{u'}(w) \\
&= \frac{1}{ (\ell p)^{m-1}   } \sum_{\{u\in B_{\ell p}:u_i = 0\}} \sum_{\{u'\in B_{\ell p}:u_i' = 0\}} \alpha_\omega^{u}(w^*f_0)\alpha_\omega^{u'}(f_0w) \\
&= \frac{1}{ (\ell p)^{m-1}   } \sum_{\{u\in B_{\ell p}:u_i = 0\}} \sum_{\{u'\in B_{\ell p}:u_i' = 0\}} \alpha_\omega^{u}(w^*) f_u f_{u'} \alpha_\omega^{u'}(w) \\
&= \frac{1}{ (\ell p)^{m-1}   } \sum_{\{u\in B_{\ell p}:u_i = 0\}}  \alpha_\omega^{u}(w^*) f_u f_{u} \alpha_\omega^{u}(w) \\
&= \frac{1}{ (\ell p)^{m-1}   } \sum_{\{u\in B_{\ell p}:u_i = 0\}}  \alpha_\omega^{u}(w^*w)  \\
&= e.
\end{align*}
Second,
\begin{align*}
g_0's &= \frac{1}{ \sqrt{ (\ell p)^{m-1} }  }  \sum_{ \{v\in B_{\ell p}:v_i = 0\} }\sum_{\{ u\in B_{\ell p}: u_i = 0\}} f_v(\alpha_\omega)^u(w) \\
&= \frac{1}{ \sqrt{ (\ell p)^{m-1} }  }  \sum_{ \{v\in B_{\ell p}:v_i = 0\} }\sum_{\{ u\in B_{\ell p}: u_i = 0\}} f_vf_u(\alpha_\omega)^u(w) \\
&= \frac{1}{ \sqrt{ (\ell p)^{m-1} }  } \sum_{\{ u\in B_{\ell p}: u_i = 0\}} f_uf_u(\alpha_\omega)^u(w) \\
&= \frac{1}{ \sqrt{ (\ell p)^{m-1} }  } \sum_{\{ u\in B_{\ell p}: u_i = 0\}} (\alpha_\omega)^u(w) \\
&= s.
\end{align*}
Finally, when $j\neq i$ we have,
\begin{align*}
	\| (\alpha_\omega)^j(s) - s\| &=  \frac{1}{ \sqrt{ (\ell p)^{m-1}  }  }   \left\|   \sum_{\{u\in B_{\ell p}:u_i=0\}} (\alpha_\omega)^j(\alpha_\omega)^u(w) - \sum_{\{u\in B_{\ell p}:u_i = 0\}} (\alpha_\omega)^u(w) \right\| \\
&= \frac{1}{ \sqrt{ (\ell p)^{m-1}  }  }   \left\|   \sum_{ \{ u\in B_{\ell  p}: u_i = 0, u_j = \ell p    \} } (\alpha_\omega)^u(w) - \sum_{\{ u\in B_{\ell p}: u_i=0, u_j = 0  \}} (\alpha_\omega)^u(w) \right\| \\
&\leq \frac{1}{ \sqrt{ (\ell p)^{m-1}  }  } \left(  \left\|   \sum_{ \{ u\in B_{\ell  p}: u_i = 0, u_j = \ell p    \} } (\alpha_\omega)^u(w) \right\| + \left\| \sum_{\{ u\in B_{\ell p}: u_i=0, u_j = 0  \}} (\alpha_\omega)^u(w) \right\|\right).
\end{align*}
Since the terms in the norms are orthogonal, the $C^*$-identity implies that
\begin{align*}
	 \left\|   \sum_{ \{ u\in B_{\ell  p}: u_i = 0, u_j = \ell p    \} } (\alpha_\omega)^u(w) \right\| &\leq \left\|  \sum_{ \{ u\in B_{\ell p}: u_i=0, u_j = \ell p  \} }  (\alpha_\omega)^u(w^*w) \right\|^{1/2} \leq \sqrt{ (\ell p)^{m-2} } \cdot \|w^*w\|^{1/2} \\
	 &\leq \sqrt{ (\ell p)^{m-2} }
\end{align*}
and a similar bound holds for the second norm. Therefore,
\begin{align*}
\| (\alpha_\omega)^j(s) -s \| \leq 2\cdot \sqrt{ \frac{(\ell p)^{m-2}}{(\ell p)^{m-1}}   }  = \frac{2}{ \sqrt{\ell p} } < \e.
\end{align*}
Since we work in the ultrapower, we may actually assume that
$$
(\alpha_\omega)^j(s) = s\;\;\;\;\;\;\; (j\neq i).
$$
Now define, for $0\leq k\leq m-1$,
$$
g_k := \sum_{r=0}^{\ell -1 } g'_{k+rp}
$$
and 
$$
v := \frac{1}{\sqrt{\ell}} \sum_{r=0}^{\ell -1} (\alpha_\omega)^{(rp)i}(s).
$$
One checks that $g_0, g_1,...,g_{m-1}, v$ satisfy the desired properties.
\end{proof}

\begin{lem} Let $A$ be a unital separable $C^*$-algebra with nonempty trace space $T(A)$, and $\alpha:\Z^m\to \Aut(A)$ be a group action. Suppose $\alpha$ has the approximate Rokhlin property and $A$ has property (SI). Then for each $i\in \{1,...,m\}$ and any $p\in \N$ there exist positive contractions
$$
a_0,...,a_{p-1}, b_0,...,b_{p-1},c_0,...,c_p
$$
in F(A) such that
\begin{enumerate}
\item[(i)] $(\alpha_\omega)^i (a_k) = a_{(k+1 \mod p)}$ and  $(\alpha_\omega)^i(b_k) = b_{(k+1 \mod p)} \;\;\;\;\;\; (0\leq k\leq p-1)$. 
\item[(ii)] $(\alpha_\omega)^i(c_k) = c_{(k+1 \mod p+1)} \;\;\;\;\;\; (0\leq k\leq p)$. 
\item[(iii)] $(\alpha_\omega)^{j}(a_k) = a_k$ and $(\alpha_\omega)^{j}(b_k) = b_k\;\;\;\;\;\; (1\leq j\neq i\leq m,\; 0\leq k\leq p-1)$
\item[(iv)] $(\alpha_\omega)^{j}(c_k) = c_k\;\;\;\;\;\; (1\leq j\neq i\leq m,\; 0\leq k\leq p)$
\item[(v)] $a_k a_n = b_kb_n = 0\;\;\;\;\;\; (0\leq k\neq n\leq p-1)$.
\item[(vi)] $c_kc_n = 0\;\;\;\;\;\; (0\leq k\neq n\leq p)$.
\item[(vii)] $b_k c_n = 0\;\;\;\;\;\; (0\leq k\leq p-1,\; 0\leq n\leq p)$.
\item[(viii)] $\sum_{k}a_k + \sum_{k} b_k + \sum_{k} c_k = 1_A$.
\end{enumerate}
\end{lem}
\begin{proof} 
Copy the proof of \cite[Theorem 6.4]{Lia16} verbatim, replacing \cite[Proposition 6.2]{Lia16} by the previous lemma. Note that since all ``matrix units'' $x_i$'s are fixed by $(\alpha_\omega)^j$, the entire image of $\psi:M_{\ell p + 1}(\C)\to F(A)$ is fixed by $(\alpha_\omega)^j$ (here $j\neq i$). Therefore (iii) and (iv) are satisfied.
\end{proof}

\begin{thm} \label{thm:main} Let $A$ be a unital simple separable nuclear $C^*$-algebra with nonempty trace space $T(A)$, and let $\alpha:\Z^m\to \Aut(A)$ be a group action. Suppose
\begin{enumerate}
	\item[(1)] $A$ has property (SI),
	\item[(2)] $\alpha$ is strongly outer,
	\item[(3)] $T(A)$ is a Bauer simplex with finite dimensional extreme boundary,
	\item[(4)] $\tau\circ \alpha = \tau$ for every $\tau\in T(A)$.
\end{enumerate}
Then $\rdim(\alpha) \leq 4^m-1$.
\end{thm}
\begin{proof} Write $(\hm, K,E)$ for the bundle $(\overline{A}^u, \partial_e T(A), E)$. Under the current assumptions, the induced group action $(\tilde{\alpha}, \bar{\alpha}):\Z^m \to \Aut(\hm, K, E)$ satisfies conditions (1) to (4) in Theorem \ref{thm:W-Rokhlin}. Therefore the action $(\tilde{\alpha}, \bar{\alpha})$ has the $W^*$-Rokhlin property. By Lemma \ref{lem:W-Rok-to-Approx-Rok}, $\alpha$ has the approximate Rokhlin property. The previous lemma, together with \cite[Proposition 2.8]{HWZ15} (which converts towers of different heights to towers of the same height), shows that the assumptions of Proposition \ref{prop:Zm-Rokhlin} are satisfied with $d=3$. By the same proposition we have $\rdim{(\alpha)} \leq 4^m -1$.
\end{proof}

Since in this case Property (SI) is implied by finite nuclear dimension, we obtain the following corollaries.

\begin{cor}
\label{cor:main} Let $A$ be a unital simple separable nuclear $C^*$-algebra with nonempty trace space $T(A)$, and let $\alpha:\Z^m\to \Aut(A)$ be a group action. Suppose
\begin{enumerate}
	\item[(1)] $A$ has finite nuclear dimension,
	\item[(2)] $\alpha$ is strongly outer,
	\item[(3)] $T(A)$ is a Bauer simplex with finite dimensional extreme boundary,
	\item[(4)] $\tau\circ \alpha = \tau$ for every $\tau\in T(A)$.
\end{enumerate}
Then $\rdim(\alpha) \leq 4^m-1$.
\end{cor}

\begin{cor}
Under the same assumptions as Corollary \ref{cor:main}, the crossed product $A\rtimes_\alpha \Z^m$ is a unital simple separable $C^*$-algebra of finite nuclear dimension.
\end{cor}
\begin{proof} From Theorem \ref{thm:main} and \cite[Proposition 2.4]{Sza15}, we see that the crossed product $A\rtimes_\alpha \Z^m$ has finite nuclear dimension. Simplicity follows from \cite[Theorem 3.1]{Kis81}.
\end{proof}


\section{Example: Bernoulli Actions}

In this final section we study the Bernoulli actions. We will show that for a large class of $C^*$-algebras, the $\Z^m$-Bernoulli shift has finite Rokhlin dimension. Throughout the section all tensor products are understood as minimal tensor products.

We recall the definition of Bernoulli actions. Let $A$ be a unital $C^*$-algebra and $\Gamma$ be a discrete group. Define $\bigotimes_{\Gamma} A$ to be the inductive limit of the system $\{ \bigotimes_F A : F\subseteq \Gamma \text{ finite subset} \}$, where the connecting maps are the natural inclusions. The (left) translation action of $\Gamma$ on itself gives rise to an action $\sigma:\Gamma\to \Aut(\bigotimes_\Gamma A)$, called the \emph{Bernoulli action} (or \emph{Bernoulli shift}).

The following proposition is a slight variation of \cite[Proposition 4.4]{Sat10}.

\begin{prop} \label{prop:bernoulli-shift}
Let $A$ be a unital simple $C^*$-algebra with a unique trace. Suppose $A$ is not isomorphic to the algebra of complex numbers $\C$. Then the Bernoulli action $\sigma:\Z^m\to \Aut(\bigotimes_{\Z^m} A)$ is strongly outer.
\end{prop}
\begin{proof}
For convenience we write $C := \bigotimes_{n\in \Z} A$. Let $\tau$ be the unique trace on $C$ and let $\pi_\tau$ be the GNS representation of $C$ associated to $\tau$. Let $\tilde{\sigma}\in \Aut(\pi_\tau(C)'')$ be the weak extension of the action $\sigma$. We show that for every $v\in \Z^m\setminus\{0\}$, the automorphism $\tilde{\sigma}^v$ is not inner in $\pi_\tau(C)''$.

Assume on the contrary that $\tilde{\sigma}^v = \Ad W_0$ for some unitary $W_0$ in $\pi_\tau(C)''$. Note that $\tilde{\sigma}^{nv}(W_0) = W_0$ for all $n\in \N$. Fix $\e > 0$ and let $x$ be a contraction in some finite tensor product $\bigotimes_F A \subseteq C$. There exist a finite subset $F'\subseteq  \Z^m$ and an element $w_0\in \bigotimes_{F'} A\subseteq C$ such that
$$
\| W_0 - \pi_\tau(w_0) \|_{2,\tau} < \e. 
$$
Then 
$$
\| W_0 - \pi_\tau( \sigma^{nv}(w_0) ) \|_{2,\tau} < \e
$$
for any $n\in \N$. If $n$ is sufficiently large, then $\sigma^{nv}(w_0)$ commutes with $x$. Therefore $W_0$ commutes with $\pi_\tau(x)$ up to $2\e$ in the 2-norm. Since $\e$ is arbitrary and $\tau$ is faithful (simplicity passes to inductive limits), we see that $W_0$ commutes with $\pi_\tau(x)$ exactly, and hence belongs to the center of $\pi_\tau(C)''$. Since $C$ has a unique trace, $\pi_\tau(C)''$ is a factor (the center is trivial). This shows that $\Ad W_0$ is the identity map, a contradiction.
\end{proof}

\begin{cor} \label{cor:bernoulli-shift}
Let $A$ be either the matrix algebra $M_n(\C)$ ($n\geq 2$) or the Jiang-Su algebra $\mathcal{Z}$ (see \cite{JS99}), and let $\sigma:\Z^m\to \Aut(\bigotimes_{\Z^m} A)$ be the Bernoulli action. Then $\sigma$ has finite Rokhlin dimension.
\end{cor}
\begin{proof}
In either case, $A$ is a unital simple separable nuclear $C^*$-algebra with a unique trace. Therefore the previous proposition shows that the Bernoulli action $\sigma$ is strongly outer. It is well-known that the infinite tensor product of a nuclear $C^*$-algebra remains nuclear. To apply Theorem \ref{thm:main}, it remains to check that $\bigotimes_{\Z^m} A$ has property (SI). In the case that $A= M_n(\C)$, the infinite tensor product is the UHF algebra of type $n^\infty$, which is known to absorb the Jiang-Su algebra $\mathcal{Z}$ (for example, see \cite{KW04} and \cite{Win10}). In the case $A=\mathcal{Z}$, the infinite tensor product is isomorphic to $\mathcal{Z}$ itself by \cite[Theorem 4]{JS99}, and in particular is Jiang-Su stable as well. Since Jiang-Su stability implies property (SI) \cite[Theorem 1.1]{MS12a}, the proof is complete. 
\end{proof}

The following proposition allows us to expand the class of $C^*$-algebra drastically.

\begin{prop} \label{prop:subalgebra-trick}
	Let $A$ be a unital $C^*$-algebra, $\alpha:\Z^m\to \Aut(A)$ a group action, and $B$ an $\alpha$-invariant subalgebra of $A$ with $1_A\in B$. Suppose the action is asymptotically commutative, in the sense that for any $\e > 0$ and
	any finite subset $F \subseteq A$ there exists $u\in \Z^m$ such that
	$$
	\| [\alpha^u(x), y ] \| < \e
	$$
	for all $x,y\in F$. Then $\rdim(\alpha)\leq \rdim(\alpha\vert_B)$
\end{prop}
\begin{proof}
	Assume $\rdim(\alpha\vert_B) = d < \infty$ otherwise there is nothing to prove. Fix a finite subset $F\subseteq A$, $\e > 0$, and $n\in \N$. By definition of Rokhlin dimension we have positive contractions
	$$
	\{ f_v^{(\ell)} \}_{v\in B_n}^{\ell = 0,1,...,d}
	$$
	in $B$ such that
	\begin{enumerate}
		\item[(1)] $\| \alpha^w(f_v^{(\ell)}) - f_{(v+w \mod n\Z^m)}^{(\ell)} \| < \e  \;\;\;\;\;\; (0\leq \ell \leq d,\; v\in B_n,\; w\in \Z^m)$;
		\item[(2)] $\| f^{(\ell)}_v f^{(\ell)}_{v'} \| < \e\;\;\;\;\;\; (0\leq \ell \leq d, \;v, v'\in B_n,\; v\neq v')$;
		\item[(3)] $\left\|  \sum_{\ell=0}^d\sum_{v\in B_n}f_v^{(\ell)} - 1_A \right\| < \e$.
	\end{enumerate}
	Since we assume the action is asymptotically commutative, there is some $u\in \Z^m$ such that
	$$
	\|  [\alpha^u(f_v^{(\ell)}) , a  ] \| < \e
	$$
	for all $v\in B_n$, $\ell=0,1,...,d$, and $a\in F$. Now if we take $g_v^{(\ell)} := \alpha^u(f_v^{(\ell)})$, then it is readily verified that $\{ g_v^{(\ell)}\}_{v\in B_n}^{\ell = 0,1,...,d}$ form the Rokhlin towers for the original action $\alpha$.
\end{proof}

Let $A$ be a unital $C^*$-algebra, and $\sigma:\Z^m\to \Aut(\otimes_{\Z^m}A)$ be the Bernoulli action. Then $\sigma$ is asymptotically commutative in the sense as in Proposition \ref{prop:subalgebra-trick}. Indeed, by definition of infinite tensor products, every element can be approximated by an element which is nontrivial only on finitely many tensor factors. Therefore the Bernoulli action would eventually shift the elements far enough and make things commutative. This, together with Proposition \ref{prop:subalgebra-trick}, gives us the following corollary:

\begin{cor} \label{lem:bernoulli}
Let $A$ be a unital $C^*$-algebra and $B\subseteq A$ be a $C^*$-subalgebra with $1_A\in B$. Let $\sigma:\Z^m\to \Aut(\bigotimes_{\Z^m} A)$ be the Bernoulli action. Then $\rdim(\sigma) \leq \rdim(\sigma|_{\bigotimes_{\Z^m}B})$.
\end{cor}
\begin{proof}
	It is enough to observe that $\otimes_{\Z^m}B$ is an invariant subalgebra of $\otimes_{\Z^m} A$ under the Bernoulli action, and then apply the previous proposition.
\end{proof}

\begin{example} Let $A$ be a unital $C^*$-algebra. Assume one of the following holds:
\begin{enumerate}
	\item[(1)] $A$ contains a unital copy of the Jiang-Su algebra $\mathcal{Z}$; for example, when $A$ is Jiang-Su stable.
	\item[(2)] $A$ contains a unital copy of a matrix algebra $M_n(\C)$ for some $n\geq 2$; for exmaple, the set of all continuous matrix-valued functions $C(X, M_n(\C))$ on some compact Hausdorff space $X$.
\end{enumerate}
Then the previous corollary, together with Corollary \ref{cor:bernoulli-shift}, implies that the Bernoulli action $\sigma:\Z^m\to \Aut(\bigotimes_{\Z^m} A)$ has finite Rokhlin dimension.
\end{example}

\bibliographystyle{plain}
\addcontentsline{toc}{chapter}{Bibliography}
\bibliography{Biblio-Database}

\end{document}